\documentclass[reqno]{amsart}

\usepackage{graphicx}
\usepackage{amssymb}
\usepackage{mathabx}
\usepackage[all]{xy}
\usepackage{color}
\usepackage{tikz-cd}
\usepackage[T1]{fontenc} 

\usepackage{lineno}
\setpagewiselinenumbers

\theoremstyle{plain} \numberwithin{equation}{section}
\newtheorem{thm}{Theorem}[section]

\newtheorem{cor}[thm]{Corollary}

\newtheorem{conj}{Conjecture}
\newtheorem{lem}[thm]{Lemma}
\newtheorem{prop}[thm]{Proposition}

\newtheorem*{clm*}{Claim}

\theoremstyle{definition}
\newtheorem{defn}[thm]{Definition}

\newtheorem{rmk}[thm]{Remark}
\newtheorem{example}[thm]{Example}

 \newcommand{\im}{\mathop{\rm Im}\nolimits}
  \newcommand{\Hom}{\mathop{\rm Hom}\nolimits}

 \newcommand{\hF}{{\widehat{F}}}
 \newcommand{\hFk}[1]{\hF\langle #1\rangle}

 \newcommand{\NN}{{\mathbb{N}}}
 \newcommand{\ZZ}{{\mathbb{Z}}}
 \newcommand{\QQ}{{\mathbb{Q}}}
 \newcommand{\RR}{{\mathbb{R}}}
 \newcommand{\pr}{{\mathrm{pr}}}

 \renewcommand{\check}{\widecheck}
 \newcommand{\ignore}[1]{} 
 \newcommand{\limslender}{$\lim$-slender{}}
 \newcommand{\nslender} {$n$-slender} 
 \newcommand{\slender}{slender} 

 \newcommand{\HEG}{{\mathbb H}}

 \newcommand{\Cc}{{\mathcal C}}
 \newcommand{\HH}{\mathcal{H}}

\newcount\hh
\newcount\mm
\mm=\time
\hh=\time
\divide\hh by 60
\divide\mm by 60
\multiply\mm by 60
\mm=-\mm
\advance\mm by \time
\def\hhmm{\number\hh:\ifnum\mm<10{}0\fi\number\mm}

\begin{document}

\title[HomLim / \today, compiled at: \ \hhmm] {Inverse limit slender groups}


\author[G. Conner]{Gregory R. Conner$^1$}
\address{
Department of Mathematics,
Brigham Young University,
Provo, UT 84602, USA}
\email{conner@mathematics.byu.edu}

\author[W. Herfort]{Wolfgang Herfort$^2$}
\address{
Institute for Analysis and Scientific Computation
Technische Universit\"at Wien
Wiedner Hauptstra\ss e 8-10/101}
\email{wolfgang.herfort@tuwien.ac.at}

\author[C. Kent]{Curtis Kent$^3$}
\address{
Department of Mathematics,
Brigham Young University,
Provo, UT 84602, USA}
\email{curtkent@mathematics.byu.edu}

\author[P. Pave\v si\'c]{Petar Pave\v si\'c$^4$}
\address{
Faculty of Mathematics and Physics
University of Ljubljana
Jadranska 21
Ljubljana, Slovenia}
\email{petar.pavesic@fmf.uni-lj.si}

\thanks{$^1$ Supported by Simons Foundation Collaboration Grant 646221.}
\thanks{$^2$ The second author is grateful for the warm hospitality at the Mathematics Department of BYU in February 2018.}
\thanks{$^3$ Supported by Simons Foundation Collaboration Grant 587001.}
\thanks{$^4$ Supported by the Slovenian Research Agency program P1-0292 and grants N1-0083, N1-0064}

\date{\today}

\begin{abstract}
Classically, an abelian group $G$ is said to be \emph{slender} if every homomorphism from the countable product
$\ZZ^\NN$ to $G$ factors through the projection to some finite product $\ZZ^n$. Various authors have
proposed generalizations to non-commutative groups, resulting in a plethora of similar but not completely
equivalent concepts. In the first part of this work we present a unified treatment of these concepts
and examine how are they related. In the second part of the paper we study slender groups in the context
of co-small objects in certain categories, and give several new applications including the proof that
certain homology groups of Barratt-Milnor spaces are cotorsion groups and a universal coefficients
theorem for \v{C}ech cohomology with coefficients in a slender group.
\end{abstract}

\vspace*{-2cm}
\maketitle

\section{Introduction}
\label{sec:Introduction}

In this introductory section we describe prior work on slender groups and give a brief summary of our
contributions.

\subsection{Prior work}

Origins of the theory of slender slender groups can be traced back to Baer's \cite{Baer1937} proof
that the group $\ZZ^\NN$ is not free and Specker's  \cite{Specker1950} remarkable
result that every homomorphism from $\ZZ^\NN$  (and from many of its subgroups) to $\ZZ$ factors through a
projection to some finite product $\ZZ^n$. This may be restated as
$$\Hom(\ZZ^\NN,\ZZ)\cong \bigoplus_{n\in\NN}\ZZ,$$
i.e., the dual of a countable product of groups $\ZZ$ is their countable sum.  A few years later
{\L}o\'{s} called a group $G$ \emph{slender} if it is torsion-free abelian, and if every homomorphism
from $\ZZ^\NN$ to $G$ factors through some finite product (see Fuchs \cite{Fuchs1958}).
He also proved that the class of slender groups is closed under subgroups, extensions and arbitrary
direct sums, and that groups containing
$\QQ$ are not slender. In particular, for a slender group $G$ we have the relation
\begin{equation}\tag{$\ast$}
\Hom(\ZZ^\NN,G)\cong \bigoplus_{n\in\NN}G.
\end{equation}
S\k{a}siada \cite{Sasiada1959} proved that a countable abelian group
is slender if, and only if, it is torsion-free and does not contain a copy of $\QQ$. Finally, without
restrictions on the cardinality,
Nunke \cite{Nunke1961} showed that an abelian group is slender if, and only if,
it is torsion-free and does not contain $\QQ$,  $\ZZ^\NN$, or $p$-adic integers for any $p$.
The theory of slender abelian groups has a natural generalization to modules over arbitrary rings, and
there is a wide body of work on slender modules and rings,
but in this paper we will be mainly interested into generalizations of slenderness to non-commutative groups.

G\"obel \cite{Gobel1975} studied non-abelian groups $G$ that satisfy {\L}o\'{s}'s definition, and proved
that $G$ is slender if, and only if, every abelian subgroup of $G$ is slender. However, the generalization
of formula $(\ast)$ is not immediate, because for non-commutative
groups $\Hom(\ZZ^\NN,G)$ is only a pointed set. G\"obel considered countable products of
arbitrary groups and showed that $G$ is slender if, and only if, the set $\Hom(\prod_{n\in\NN} A_n,G)$
can be expressed as a subset of the product $\prod_n\Hom(A_n,G)$, consisting of elements with finite support
(see \cite[Theorem 4.3]{Gobel1975} for precise formulation).

A further step ahead was
taken by Eda \cite{Eda1992} who defined \emph{non-commutatively slender} (or \emph{n-slender}) groups by
replacing the product $\ZZ^\NN$ and its projections to finite factors $\ZZ^n$ with the so-called free complete
product of cyclic groups, denoted $\HEG$, and its projections to free groups on $n$-letters $F_n$.
Group $\HEG$ has a convenient topological description as the fundamental group of the Hawaiian earring
(a countable one-point union of a sequence of circles whose radii decrease to zero), and the projections
to $F_n$ are induced for each $n$ by the projection of the Hawaiian earring to the first $n$ circles
(see \cite[Appendix]{Eda1992}). Eda proved that n-slender groups are torsion-free and are
preserved by taking subgroups, free products and restricted direct products. He also showed
that n-slender is a strictly stronger condition than slender, but that
for abelian groups the two concepts coincide.

Slenderness can be also seen as an automatic continuity condition in the sense of Dudley \cite{Dudley1961}.
Let us endow $\ZZ$ and $G$ with the discrete topology and $\ZZ^\NN$ with the product topology.
Then it is easy to see that $G$ is slender precisely when every homomorphism
from $\ZZ^\NN$ to $G$ has open kernel (in other words, it is continuous when $G$ is endowed with the discrete topology). This lead
Conner and Corson \cite{ConnerCorson2019} to define a discrete group $G$ to be
\emph{lch-slender} (resp. \emph{cm-slender}) if every homomorphism from a locally compact Hausdorff
(resp. completely metrizable) group $A$ to $G$ has open kernel. They showed that many interesting
groups, including free groups, free abelian groups and torsion-free word-hyperbolic groups,
are lch-slender and cm-slender and that for abelian groups we recover the usual slender groups.
Several types and properties of automatically continuous were studied in
\cite{CorsonKazachkov2020, Kryuchkov2007}.
Most recently, Corson and Varghese \cite{CorsonVarghese2020} strengthened the connection between
slender groups and automatic continuity by showing that a group is lch-slender if it is torsion-free
and it does not contain $\QQ$ or $p$-adic integers for any $p$.

\subsection{Our contribution}

A common feature of the above  examples is that the groups used to test
slenderness are related to inverse limits: $\ZZ^\NN$ can be represented as an inverse limit of
$\ZZ^n$, $\HEG$ is a dense subgroup of the inverse limit of free groups $F_n$, every connected
locally compact Hausdorff group is an inverse limit of Lie groups.
This (and some isolated results for maps from inverse limits) motivated our study of
groups that are slender with respect to inverse limits of groups.

Recall that every inverse sequence of groups
$$H_1\longleftarrow H_2 \longleftarrow H_3 \longleftarrow \cdots$$
gives rise to a limit group $\varprojlim H_i$, together with projections $p_i\colon \varprojlim H_i\to H_i$.
A group $G$ is \emph{inverse limit slender} (or \emph{lim-slender}) if for every inverse sequence
with surjective bonding maps, every homomorphism from $\varprojlim H_i$ to $G$ factors through some
projection $p_i$.

Our first result states that slenderness with respect to inverse limits can be checked on a single
test group. Let $\hF$ denote the inverse limit of an inverse sequence of finite rank free groups
(see Definition \ref{def:hF} for a formal description). This characterization allows us to prove that
inverse limit slender groups are preserved by several basic group constructions.

\textbf{Theorem 1.} (Theorem \ref{thm:lim-slender} and Theorem \ref{thm:closed under})\label{thm:1} \emph{
A group $G$ is inverse limit slender if, and only if, every homomorphism from $\hF$ to $G$ factors through
a projection to a free group of finite rank.\\[1mm]
The class of inverse limit slender groups is closed under taking subgroups and extensions,
and under directs sums and free products of arbitrary cardinality.}

The relation between different types of slender groups is given by the following theorem.

\textbf{Theorem 2.} (Theorem \ref{thm:relations between slender} and Corollary \ref{cor:equivalent for abelian})
\emph{The class of n-slender groups is contained in the class of inverse limit slender groups,
the class of inverse limit slender groups is contained in the class of slender groups, and
the class of slender groups is contained in the class of lch-slender groups. All containments are strict.\\[1mm]
Furthermore, among abelian groups the n-slender, inverse limit slender and slender groups coincide.}

Interestingly, all examples that distinguish various versions of slenderness are uncountable groups, so we
state the following

{\bf Conjecture:} \emph{For countable groups the classes of  slender groups, n-slender groups, inverse
limit slender groups, cm-slender groups and lch-slender groups coincide.}

Inverse limit slender groups bring to light a connection with an important concept from category
theory. In a category $\Cc$ that admits direct limits an object $X$ in $\Cc$ is said to be \emph{small}
if for every direct sequence with monomorphic bonding maps $(X^i)$ there exists a natural bijection
$\varinjlim \Hom(X,X^i)=\Hom(X,\varinjlim X^i)$. For example, finitely generated groups are
small object in the category of groups. Dually, a \emph{co-small} object in a category $\Cc$ that admits
inverse limits is an object $X$, such that for every inverse sequence $(X_i)$ with epimorphic bonding maps
there is a natural bijection $\varinjlim \Hom(X_i,X)=\Hom(\varprojlim X_i,X)$. Clearly,
a group $G$ is inverse limit slender if, and only if it is a co-small object in the category of groups.
In Propositions \ref{prop:slender as co-small} and \ref{prop:n-slender as co-small} we give analogous
characterizations for slender and n-slender groups.

One can obtain a higher dimensional analogue of the Hawaiian earring by taking for each $n$ the space
$\HH_n$ defined as a countable one-point union of $n$-dimensional spheres in $\RR^{n+1}$ whose radii
decrease to zero, i.e. the one-point compactification of a countable collection of open $n$-balls.
These spaces were first studied by Barratt and Milnor \cite{BarrattMilnor1962} who proved a remarkable
fact that, although $\HH_n$ is $n$-dimensional, there are infinitely many indices $i$ for which
the singular homology groups $H_i(\HH_n)$ are uncountably infinite.
As a step toward a precise description of these homology groups we prove the following\\[1mm]
\textbf{Theorem 3.} (Theorem \ref{thm: cotorsion})
\emph{For $l>n>1$, the image of the Hurewicz map $\pi_l(\HH_n)\to H_l(\HH_n)$ is cotorsion.
In particular, $H_{n+1}(\HH_n)$ is  cotorsion for every $n>1$.}

Together with Barratt and Milnor's results this proves that $H_3(\HH_2)$ is an uncountable cotorsion group. The proof of the above result is based on the ability to choose arbitrarily small representatives
of homology classes (see Lemma \ref{lem: isom}), so we believe that a more general result is true:

{\bf Conjecture:} \emph{For $l>n>1$, the groups $H_l(\HH_n)$ are cotorsion.}

As mentioned before, part of our motivation for this work is of topological nature.
In particular, our study of inverse limits of covering spaces
naturally requires good understanding of homomorphisms from inverse limits groups
(see Conner, Herfort, Kent and Pavesic \cite{ConnerHerfortKentPavesic1, ConnerHerfortKentPavesic2}
for further information). In the final part of this paper we present two applications of our work
on (shape) homotopy groups and (\v Cech) cohomology groups of Peano continua
(see Section \ref{sec:Applications to shape groups and Cech cohomology} for relevant definitions).

A well-known theorem of Shelah \cite{Shelah1988} states that the fundamental group of a Peano continuum
is either finite presented or uncountable and this implies that the shape fundamental group $\check\pi_1(X)$
of a Peano continuum $X$ may be very complicated. In fact, it is often uncountable and can even be locally free
with uncountable free subgroups. Nevertheless, we will show that $\check\pi_1(X)$ admits relatively
few homomorphisms into slender groups.

\textbf{Corollary 4.} (Corollary \ref{cor:hom check})
\emph{If $X$ is a Peano continuum and $G$ is a a countable inverse limit slender group, then there are only
countably many homomorphisms from the first shape group $\check \pi_1(X)$ to $G$.}

Our final  application is a Universal Coefficients Theorem that expresses \v Cech cohomology groups
with coefficients in a slender abelian group in terms of \v Cech homology groups. The theorem relies
on the equivalence between slender and inverse limit slender abelian groups.

\textbf{Theorem 5.} (Theorem \ref{thm: uct}) \emph{
If $G$ is a slender abelian group, then there is a short exact sequence }
$$ 0 \longrightarrow \lim\limits_{i \rightarrow\infty} \operatorname{Ext}\bigl(H_{n-1}(X_i),G\bigr)
\longrightarrow \check{H}^n(X;G) \longrightarrow \Hom(\check{H}_n(X), G)\longrightarrow 0.$$

\subsection{Outline}
In Section \ref{sec:Preliminaries on inverse and direct limits} we introduce the notation and terminology
of categorical limits. Section \ref{sec:Inverse limit slender groups} contains  the main
results on inverse limit slender groups and on their relation to other types of slender groups. In Section
\ref{sec:Slender as a co-small object} we examine categorical aspects of slender groups, viewed
as co-small objects in corresponding categories. In Section \ref{sec:Homology of Barratt-Milnor examples}
we study algebraic compactness and
the cotorsion property of homology groups of Barratt-Milnor spaces. Finally, in Section
\ref{sec:Applications to shape groups and Cech cohomology} we give some applications of inverse limit
slender groups to shape fundamental groups and \v Cech cohomology groups.

\section{Preliminaries on inverse and direct limits}
\label{sec:Preliminaries on inverse and direct limits}

In this section we set the basic notation on direct and inverse sequences and their limits, and refer
to Geoghegan \cite[Section 11.2]{Geoghegan2008} for a more detailed exposition.

\begin{defn}  Let $\Cc$ be a category. An \emph{inverse sequence} $(X_i, p_i^j)$ in $\Cc$ consists
of a set of objects $\{X_i\mid i\in\NN\}$ in $\Cc$ and a set of morphisms (called \emph{bonding maps})
$\{p_i^j\colon X_j\to X_i\mid i<j\}$ in $\Cc$, satisfying $p_i^j\, p_j^k=p_i^k$ whenever $i<j<k$.
We often depict an inverse system as a diagram
$$X_1\stackrel{p_1^2}{\longleftarrow} X_2\stackrel{p_2^3}{\longleftarrow} X_3 \longleftarrow \cdots$$
In the categories that are of interest to us (sets, pointed sets, abelian groups, groups, topological spaces,
pointed topological spaces, topological groups)
every inverse sequence $(X_i, p_i^j)$ has an \emph{inverse limit}, i.e. an object
$\varprojlim (X_i, p_i^j)$ in $\Cc$, together with morphisms $p_i\colon \varprojlim (X_i, p_i^j)\to X_i$ in $\Cc$
called \emph{projections} satisfying $p_i=p_i^j p_j$ whenever $i<j$ and the usual universal property
(see \cite[p.235]{Geoghegan2008}). In fact, in the above mentioned categories we have standard models
for the inverse limit whose underlying set is given
$$\varprojlim ( X_j, p_i^j)= \Bigl\{(x_i)\in \prod_{i\in \NN}  X_i \mid x_i=p_i^j(x_j) \Bigr\}$$
and projections $p_i$ are obtained by projecting each sequence to its $i$-th component. For categories
with algebraic or topological structure
the underlying sets are equipped with component-wise algebraic operations or/and product topology.
\end{defn}

We next define a special,  well-behaved class of inverse system.

\begin{defn}
An inverse sequence  $\{ X_j, p_i^j\}$ satisfies \emph{Mittag-Leffler} condition (or \emph{is an ML-sequence})
if for every $i$ there exists $j > i$ such that $p_i^k(X_k) = p_i^j(X_j)$ for every $k \geq j$. 
\end{defn}

Clearly, every inverse sequence with surjective bonding maps is an ML-sequence, On the other hand,
a sequence, whose bonding maps are injective is an ML-sequence only if it is eventually constant.

Inverse sequences and their limits have categorical duals in the form of direct sequences and direct limits.

\begin{defn}  A \emph{direct sequence} $(X^i, u_i^j)$ in a category $\Cc$ consists
of objects $\{X^i\mid i\in\NN\}$ and bonding maps $\{u_i^j\colon X^i\to X^j\mid i<j\}$ in $\Cc$, satisfying $u_j^k\, u_i^j=u_i^k$ whenever $i<j<k$. The corresponding diagram is
$$X^1\stackrel{u_1^2}{\longrightarrow} X^2\stackrel{u_2^3}{\longrightarrow} X^3 \longrightarrow \cdots$$
A \emph{direct limit} of such sequence is an object
$\varinjlim (X^i, u_i^j)$ in $\Cc$, together with morphisms $u_i\colon X^i\to\varinjlim (X^i, u_i^j)$,
satisfying  $u_j=u_i u_i^j$ whenever $i<j$ and the usual universal property
(see \cite[p.238]{Geoghegan2008}). In all of the above-mentioned categories a model for the direct limit
can be obtained by taking a disjoint union of underlying sets, modulo certain equivalence relation,
and endowing it with a suitable algebraic or topological structure:
$$\varinjlim (X^i, u_i^j)= \coprod_{i\in \NN}  X^i\big/\sim,$$
where $\sim$ is the equivalence relation, generated by  $x_i\sim u_i^j(x_i)$ for all $i\in\NN$ and $x_i\in X^i$.
In particular, if all bonding maps in a direct sequence are injective, then its direct limit is simply
the union of its terms.
\end{defn}

The following special case will be of central importance for our discussion.

\begin{defn}\label{def:hF}
Let $\{x_n\mid n\in\NN\}$ be a countable set, and for every $i\in\NN$ let $F_i$ be
the free group on the alphabet $\{x_1, \cdots, x_i\}$. For $i<j$ there are natural inclusions
$u_i^j\colon F_i\to  F_j$ and projections $p_i^j\colon F_j\to F_i$, which are related by
$p_i^j u_i^j=1_{F_i}$. Thus we obtain a direct system $(F_i,u_i^j)$ and an
inverse system $(F_i,p_i^j)$. The direct limit of the former is $F_\infty:=\varinjlim (F_i,u_i^j)$, the free
group generated by $\{x_n\mid n\in\NN\}$, while the inverse limit of the latter is
$\hF=\varprojlim (F_i,p_i^j)$, which is sometimes called \emph{unrestricted free product} of infinite
cyclic groups. In addition, we will denote by $\hFk{k}$ the subgroup of $\hF$ obtained as inverse limit
of free groups on the subset $\{x_n\mid n>k\}$, so that $\hFk{0} =\hF$ and
$\hFk{k}\subseteq \ker p_k$.
\end{defn}
Note that, the relation between maps $u_i^j$ and $p_i^j$ implies that the group $F_\infty$ can be
naturally
identified with a subgroup of $\hF$. Even more, if we endow $\hF$ with the product topology, then
$F_\infty$ is a countable dense subgroup of $\hF$.

We conclude this section by recalling certain commutativity relations between direct and inverse limits.
Let $\varprojlim X_i$ be the limit of an inverse sequence $(X_i,p_i^j)$ in a category $\Cc$.
We can associate to every morphism $f\colon A\to \varprojlim X_i$ the sequence
$(p_if)\in \prod \Hom(A,X_i)$. One can easily check that this actually determines a natural bijection
$$\Hom(A,\varprojlim X_i)\stackrel{\cong}{\longrightarrow}\varprojlim \Hom(A,X_i).$$
Dually, by associating to a morphism $f\colon\varinjlim X^i\to A$ the sequence of morphisms
$(f u_i)\in \prod \Hom(X^i,A)$, we obtain a natural bijection
$$\Hom(\varinjlim X^i, A)\stackrel{\cong}{\longrightarrow}\varprojlim \Hom(X^i,A).$$

Analogous relations for morphisms
into direct limits and from inverse limits are more complicated.
Consider a direct sequence $(X^i,u_i^j)$ in a category $\Cc$.
For every object $A$ in $\Cc$ we may apply the covariant functor $\Hom(A,-)$ and obtain a direct sequence
of sets $\big(\Hom(A,X^i),(u_i^j)_*\big)$. Every element of $\varinjlim\big(\Hom(A,X^i),(u_i^j)_*\big)$
is represented by some $h\in\Hom(A,X^i)$. It is easy to check that the correspondence
$h\mapsto u_i\circ h\in \Hom(A, \varinjlim (X^i,u_i^j))$ defines a function
$$\Delta_A\colon \varinjlim\big(\Hom(A,X^i),(u_i^j)_*\big)\longrightarrow
\Hom(A,\varinjlim (X^i,u_i^j)).$$
Note that $u$ is injective if bonding maps $u_i^j$ are monomorphisms in $\Cc$.

Dually, if $(X_i,p_i^j)$ is an inverse sequence in $\Cc$, then for every object $A$ in $\Cc$ we have
a  direct sequence of sets $\big(\Hom(X_i,A),(p_i^j)^*\big)$. By assigning to every $h\in\Hom(X_i,A)$
the element $h\circ p_i\in \Hom(X_i,A)$ we obtain a function
$$\nabla_A\colon \varinjlim\big(\Hom(X_i,A),(p_i^j)^*\big)\longrightarrow
\Hom(\varprojlim (X_i,p_i^j),A),$$
which is injective if bonding maps $p_i^j$ are epimorphisms in $\Cc$.

\section{Inverse limit slender groups}
\label{sec:Inverse limit slender groups}

The main objective of this section is to describe the properties of inverse limit slender groups and relate
it to other types of slender groups. We begin by recalling the classical definition of slender groups
and by showing that the integers are a slender group.

\begin{defn}[cf. Fuchs {\cite[Ch. XIII]{Fuchs2015}}]
Let $\ZZ^\NN$ be the countable product of infinite cyclic groups. A group $G$ is \emph{\slender} if every
homomorphism $\varphi: \ZZ^\NN\to G$ can be factored through some $\pr_n\colon\ZZ^\NN\to \ZZ^n$
(projection to the first $n$ factors of the product), i.e., if there exists
$n\in \NN$ and a homomorphism $\varphi'\colon\ZZ^n\to G$ such that the following diagram commutes
$$\xymatrix{
\ZZ^\NN \ar[rr]^\varphi \ar[dr]_{\mathrm{pr}_n} & & G\\
& \ZZ^n \ar[ur]_{\varphi'}}
$$
\end{defn}

For the sake of completeness, we give a short, self-contained proof of the well-known fact that $\mathbb Z$
is slender.

\begin{prop}
 $\ZZ$ is slender.
\end{prop}
\begin{proof}
If there is a homomorphism $\varphi: \ZZ^{\NN} \to \ZZ$ that does not factor through any finite projection,
then for every $i$ there exists $b_i\in \prod_{n\leq i}\{0\} \times \prod_{n>i} \ZZ$ such that
$\varphi(b_i)\neq 0$. One can easily check that the formula $\psi\bigl((x_i)\bigr):= \sum_i x_ib_i$
defines a homomorphism $\psi: \ZZ^{\NN} \to\ZZ^{\NN}$, such that
 $\varphi(\psi(\mathbf e_n)) \neq 0$ for all $n$.

Let $a_1 = 1$, and define $a_2, a_3, \cdots$ recursively by letting $a_{i+1}$ be the minimal multiple
of $a_i$ that does not divide $\varphi(\psi(a_i\mathbf{e}_i))$. The set $\bigl\{(x_i)\in\ZZ^\NN \mid x_i\in
\{0, a_i\}\bigr\}$ is uncountable, so there exist sequences $(x_i) \neq (y_i)$ such that
$\varphi\bigl(\psi\bigl((x_i)\bigr)\bigr) = \varphi\bigl(\psi\bigl((y_i)\bigr)\bigr)$.
Let $n$ be the minimal index for which $x_n \neq y_n$.  Then $(x_i) - (y_i) \pm a_n\mathbf e_n$ is divisible by
$a_{n+1}$ but $\varphi\bigl(\psi\bigl((x_i)- (y_i)\pm a_n\mathbf e_n\bigr)\bigr)= \varphi\bigl(
\psi(a_n\mathbf e_n)\bigr)$ is not, a contradiction.
\end{proof}

One can clearly recognize a recurring pattern that will appear throughout the paper: a group is slender
if it cannot certain infinitary operations. Let us now define a type of slenderness that is
of principal interest to this work.

\begin{defn}
A group $G$ is \emph{inverse limit-slender} (occasionally abbreviated to \emph{\limslender}) if for
every inverse sequence of groups $(H_i,p_i^j)$ satisfying the
Mittag-Leffler condition and for every homomorphism $\varphi\colon \varprojlim (H_i,p_i^j)\to G$ there
exists a homomorphism $\overline\varphi\colon H_i\to G$, such that $\varphi=\overline\varphi\ p_i$.
In other words, every homomorphism from $\varprojlim H_i$ to $G$ factors through some term in
the sequence as in the following diagram
$$\xymatrix{
\varprojlim H_i \ar[rr]^\varphi \ar[dr]_{p_i} & & G\\
& H_i \ar[ur]_{\overline\varphi}}
$$
\end{defn}

The following example shows  that without the assumption that the inverse sequences is Mittag-Leffler
even the integers as the archetypal slender group would not satisfy the defining property for
a \limslender\  group.

\begin{example}\label{exm: example}
For each prime $p$, let $\mathbb Q_{(p)} = \{a/b\in \mathbb Q \mid p\nmid b\}$ and consider the following
inverse sequence of groups:  $H_i = \bigcap\limits_{p\leq i} \mathbb Q_{(p)}$ and
$p_i^j: G_j \to G_i$ is the inclusion.  Then $\varprojlim H_i = \mathbb Z$, but the identity map
$\varprojlim H_i \to \ZZ$ clearly cannot factor through any $H_i$. In general, if the sequence is
not Mittag-Leffler, then its terms  can have relations that do not hold in the inverse
limit. In our example every element of every factor can be divided by arbitrarily large integers $n$,
which is not the case in $\ZZ$.
\end{example}

To proceed, we need the following technical lemma. Recall from definition \ref{def:hF} that the group
$\hF$ contains the free group $F_\infty$ generated by elements $\{x_n\mid n\in\NN\}$.

\begin{lem}\label{lem: not slender}
Let $(H_i, p_i^j)$ be an inverse sequence with surjective bonding maps.  If a homomorphism
$\varphi: \varprojlim H_i \to G$ does  not factor through any projection $p_i$, then there exists a homomorphism
$\psi: \hF \to \varprojlim H_i$ such that $\varphi(\psi(x_i))\ne 1$ for all $i$.
In particular, $\varphi\circ \psi\colon\hF\to G$ does not factor through any projection $\hF\to F_i$.
\end{lem}

\begin{proof}
Observe that the kernels $\ker (p_i\colon\varprojlim H_i\to H_i)$ form a decreasing sequence of nested normal
subgroups of $\varprojlim H_i$ and that $\varphi$ factors through $p_i$ if, and only if,
$\ker(p_i)\subseteq \ker(\varphi)$. Thus, by our assumptions, $\ker p_i-\ker\varphi\ne \emptyset$ for all $i$.

Pick $u_1\in \ker(p_1)- \ker(\varphi)$.  Since $u_1\ne 1$ (because $\varphi(u_1)\ne 1$), there exists a
minimal index $k_1$, such that $p_{k_1}(u_1) \neq 1$. Pick $u_2\in \ker(p_{k_1})- \ker(\varphi)$. As before,
there exists a minimal $k_2$, such that $p_{k_2}(u_2) \neq 1$. By continuing the process we obtain
a sequence of elements $u_i\in\ker p_{k_{i-1}}-\ker\varphi$, $i=1,2,\ldots$, satisfying $p_{k_i}(u_i)\ne 1$.

For each $n$, define $\psi_n: F_n \to H_{k_n}$ by $\psi_n(x_i) = p_{k_n}(u_i)$ for $i\in\{1, \cdots, n\}$.
It is then easy to check that the diagrams
$$\xymatrix{
F_n\ar[d] \ar[r]^{\psi_n} & H_{k_n} \ar[d]^{p^{k_n}_{k_{n-1}}}\\
F_{n-1} \ar[r]_{\psi_{n-1}} & H_{k_{n-1}}
}$$
commute for every $n$, so we get an induced map $\psi\colon \hF \to \varprojlim H_{k_i}=\varprojlim H_i$,
with the property that $\varphi\circ \psi (x_i)  =\varphi(u_i) \neq 1$.
\end{proof}

We have the following characterization of inverse limit slender groups

\begin{thm}\label{thm:lim-slender}
The following statements for a group $G$ are equivalent.
\begin{enumerate}
  \item\label{inv1} $G$ is inverse limit slender.
  \item\label{inv2} Every homomorphism $\varphi\colon \hF\to G$ factors through some projection $\hF\to F_i$.
  \item For every homomorphism $\varphi\colon \hF\to G$, there exists $i \in\NN$, such that $\varphi(x_j)=1$
  for $j>i$.
  \item\label{inv3} For every (non-necessarily Mittag-Leffler) inverse sequence $(H_i,p_i^j)$ and
  every homomorphism $\varphi\colon \varprojlim H_i\to G$ there exists an index $i$ such that $\varphi$ factors
  through the projection $\varprojlim H_i\to p_i(\varprojlim H_i)$.
\end{enumerate}
\end{thm}
\begin{proof}
(1) implies (2) because the inverse sequence that defines $\hF$ is Mittag-Leffler.
Also, (2) trivially imply (3).

That (3) implies (4)
follows from lemma \ref{lem: not slender}. In fact, $\lim H_i$ can be viewed as the inverse limit of the
sequence $(p_i(\varprojlim H_i),p_i^j)$ with surjective bonding maps. If there exists a homomorphism
$\varphi\colon\varprojlim H_i\to G$ that does not factor through any projection to $p_i(\varprojlim H_i)$,
then by lemma \ref{lem: not slender} there exists a homomorphism from $\hF$ that does not satisfy
condition (3).

Finally, to show that (4) implies (1) consider a Mittag-Leffler sequence $(H_i, p_i^j)$ and a homomorphism
$\varphi\colon \varprojlim H_i\to G$. By (3) $\varphi$ factors through the projection to
some $p_i(\varprojlim H_i)$. By the Mittag-Leffler property there exists some $j>i$, such that
$p_i^j(H_j) =p_i(\varprojlim H_i)$ but that means that $\varphi$ factors through the projection
$p_j$.
\end{proof}







We now show that inverse limit slenderness is preserved by several basic operations.

\begin{thm}
\label{thm:closed under}
The class of inverse limit slender groups is closed under taking subgroups, extensions, directs sums and free
products.
\end{thm}
\begin{proof}
That a subgroup of a {\limslender} group is also {\limslender}  is obvious.

To see that an extension of {\limslender} groups is \limslender, let $A$ be a normal subgroup  of $G$
with $q\colon G\to G/A$ the projection to the quotient group,
and assume that both $A$ and $G/A$ are \limslender. Clearly, for every homomorphism
$\varphi\colon \hF\to G$
there exists $n$ such that $\hFk{n} \subseteq \ker (q\circ\varphi)$, therefore $\varphi(\hFk{n})\subseteq A$.
Since $A$ is also slender, there exists $m\ge n$, such that $\varphi(x_i)=1$ for $i\ge m$.

It is obvious that a finite direct sum of \limslender{} groups is trivially \limslender.
We will show that a direct sum of an arbitrary family $\{G_\lambda\mid \lambda\in\Lambda\}$
of \limslender{} groups is also \limslender.

If $\bigoplus_\Lambda G_\lambda$ is not \limslender, then there exists a homomorphism $\varphi\colon \hF\to
\bigoplus_\Lambda G_\lambda$ that does not factor through any projection. By applying
Lemma \ref{lem: not slender} if necessary, we may assume that $\varphi(x_i)\ne 1$ for all $i\in\NN$.
By definition of direct sum, there exists a finite subset $\Lambda_1\subset\Lambda$, such that
$\varphi(x_1)$ is contained in $\bigoplus_{\Lambda_1} G_\lambda$. Since
$\bigoplus_{\Lambda_1} G_\lambda$ is \limslender{}, there exists $k_1\in\NN$, such
that
$\varphi(x_i)\in\ker \big(p_{\Lambda_1}\colon \bigoplus_\Lambda G_\lambda\to \bigoplus_{\Lambda_1} G_\lambda\big)
$ for $i\ge k_1$.
Similarly, there exists a finite set $\Lambda_2$ ($\Lambda_1\subseteq \Lambda_2\subseteq \Lambda$),
such that $\varphi(x_{k_1})\in\bigoplus_{\Lambda_2} G_\lambda$, and an integer $k_2\ge k_1$, such that
$\varphi(x_i)\in\ker (p_{\Lambda_2})$ for $i\ge k_2$. Continuing this procedure, we obtain a sequence
of finite subsets $\Lambda_1\subseteq \Lambda_2\subseteq\cdots \Lambda$ and elements
$x_{k_0}=x_1,x_{k_1},x_{k_2},\ldots$ of $\hF$, such that $\varphi(x_{k_i})\in\bigoplus_{\Lambda_n} G_\lambda$
for $i\le n$, and $\varphi(x_{k_i})\in\ker (p_{\Lambda_n})$ for $i>n$.

Consider $\varphi(w)$ for $w:=x_{k_0} x_{k_1} x_{k_2}\cdots\in\hF$. There exists $n$, so that
$p_\lambda(\varphi(w))=1$ for $\lambda\in(\bigcup \Lambda_k)-\Lambda_n$. It is easy to check that for
$w':=(x_{k_0}\cdots x_{k_n})^{-1}w=x_{k_{n+1}} x_{k_{n+2}}\!\!\!\!\cdots $ and
$w'':=(x_{k_0}\cdots x_{k_{n+1}})^{-1}w=x_{k_{n+2}} x_{k_{n+3}}\!\!\!\!\cdots $ we have
$p_\lambda\varphi(w')=p_\lambda \varphi(w'')=1$ for all $\lambda\in\bigcup \Lambda_k$. But this implies
that $\varphi(x_{n+1})=\varphi(w'(w'')^{-1})=1$, which is a contradiction.

Finally, that free products of \limslender{} groups is \limslender{} follows from \cite[Theorem 1.5]{Slutsky2013}
because, when endowed with the product topology, $\hF$ becomes a complete topological group.
\end{proof}

We are now going to compare inverse limit slender groups with classical  slender group, n-slender groups,
lch-slender groups and cm-slender groups (see Introduction for the relevant definitions).

Since the group $\ZZ^\NN$ is the limit of an inverse sequence of groups $\ZZ^n$ with surjective bonding
maps, it immediately folows that every inverse limit slender group is slender. The inclusion is proper,
because the group $\hF$ is slender (by \cite[Theorem 4.3]{Gobel1975}, since $\hF$ is locally free),
but it is not inverse limit slender, because the identity homomorphism on $\hF$
does not factor through any of the projections).

To compare n-slender and inverse slender group we prove the following result.

\begin{prop}
\label{prop:nc implies inverse limit}
Every \nslender~ group is inverse limit slender.
\end{prop}
\begin{proof}
Assume $G$ is not inverse limit slender. By theorem \ref{thm:lim-slender} there exists a homomorphism
$\varphi\colon\hF \to G$, such that $\varphi(x_i)\ne 1$ for infinitely many indices $i$. But then
the restriction $\varphi|_\HEG\colon\HEG\to G$ of $\varphi$ to the free complete product of cyclic groups $\HEG$
clearly cannot factor through any projection to a free group of finite rank, so $G$ is not n-slender.
\end{proof}

On the other hand Eda \cite[Theorem 1.2]{Eda1998} proved that the group $\HEG$ is slender but it is obviously not
n-slender (because the identity homomorphism from $\HEG$ to itself does not factor through any of the
projections). Note that $\HEG$ can be viewed as a dense subgroup of $\hF$ if we equip the latter
with the inverse limit topology (as a subspace of the product of discrete groups $F_n$). Thus every
continuous homomorphism from $\hF$ to a topological group $G$ is uniquely determined by its restriction
to $\HEG$. This fact can be extended to all subgroups of $\hF$ that contain the infinitely generated
free group $F_\infty$.

\begin{thm}
\label{thm:res is injective}
Let $G$ be an inverse limit slender group and let $H$ be any subgroup of $\hF$ containing $F_\infty$.
Then the restriction function
$$\mathrm{res}\colon\Hom(\hF,G)\to\Hom(H,G)$$
is injective.
\end{thm}
\begin{proof}
Note that we can turn every group into a topological group by endowing it with the discrete
topology. If $G$ is inverse limit slender, then by \ref{thm:lim-slender}(2)
every homomorphism $\varphi\colon \hF \to G$ factors as $\varphi=\overline\varphi\circ p_i$
for some homomorphism $\overline\varphi\colon F_i\to G$. If we endow $F_i$ and $G$ with the discrete
topology and $\hF$ with the product topology, then $\overline\varphi$ is continuous because
$F_i$ is discrete and $p_i$ is continuous by definition of product topology. We conclude that $\varphi$
is continuous, being a composition of continuous functions. Since $p_i(F_\infty)=F_i$ it follows that
$F_\infty$ is dense in $\hF$, and the same holds for every subgroup $H\le \hF$ containing $F_\infty$.
Therefore, if $\varphi,\psi\colon\hF\to G$ are homomorphims satisfying $\varphi|_H=\psi_H$,  then
$\varphi=\psi$, so the restriction homomorphism is injective.
\end{proof}

Note, that $\textrm{res}\colon\Hom(\hF,G)\to\Hom(\HEG,G)$ need not be injective for an arbitrary group $G$.
In fact, the normal closure $N$ of $\HEG$ in $\hF$ is a proper subgroup of $\hF$, so the quotient
homomorphism $q\colon \hF\to\hF/N$ is a non-trivial element of  $\Hom(\hF,\hF/N)$ that restricts to
a trivial element of $\Hom(\HEG,\hF/N)$.

In a recent paper Corson and Varghese \cite{CorsonVarghese2020} proved that a group $G$ is lch-slender
if, and only if, $G$ is torsion-free and it does not contain subgroups isomorphic to $\QQ$ or to $p$-adic
integers for any prime $p$. This coincides with the characterization of \emph{stout} groups
in G\"obel \cite[Theorem 4.1]{Gobel1975}. Following G\"obel \cite{Gobel1975} a group $G$ is \emph{stout}
if $\Hom(\prod_iH_i, G)$ is trivial whenever $\{H_i\mid i\in\NN\}$ is a set of groups with trivial $\Hom(H_i,G)$.
He proved that every slender group is stout, and that $\ZZ^\NN$ is stout but not slender.

We may summarize the above discussion in the following theorem.

\begin{thm}
\label{thm:relations between slender}
The class of n-slender groups is contained in the class of inverse limit slender groups, which is contained
in the class of slender groups, which is in turn contained in the class of lch-slender groups. All
containments are strict.

Furthermore, the class of cm-slender groups is contained in the class of inverse limit slender groups.
\end{thm}

The world of slender groups is quite rich and varied, see Remark \ref{rmk:slender}.
Eda \cite[Theorem 3.3]{Eda1992} proved that an abelian group is n-slender if, and only if, it is slender.
Since the class of inverse limit slender groups is intermediate, they must all coincide.
On the other hand, the group $\ZZ^\NN$ is lch-slender by \cite{CorsonVarghese2020} but is not slender.

\begin{cor}
\label{cor:equivalent for abelian}
For abelian groups the class of slender, inverse limit slender and n-slender groups coincide, and are
strictly contained in the class of lch-slender abelian groups.
\end{cor}

Incidentally, the class of abelian lch-slender groups coincides with the class of \emph{cotorsion-free}
groups (i.e., groups that do not contain any cotorsion group in the sense of Fuchs \cite[Ch.9]{Fuchs2015}

All examples of groups that distinguish the above classes are uncountable, and we have not been able to
find analogous examples of countable groups. Thus we state the following conjecture.

\begin{conj}
The classes of  slender, inverse limit slender, n-slender, lch-slender and cm-slender groups coincide
for all countable groups.
\end{conj}

\section{Slender as a co-small object}
\label{sec:Slender as a co-small object}

The concept of inverse limit slender groups leads to a connection between slender groups and
the categorical notion of small and co-small objects that we examine in this section.

As explained in Section \ref{sec:Preliminaries on inverse and direct limits}, in any category
there are natural bijections
$$\Hom(A,\varprojlim X_i)\cong\varprojlim \Hom(A,X_i)\ \ \text{and} \ \
\Hom(\varinjlim X^i, A)=\varprojlim \Hom(X^i,A)$$
whenever suitable direct and inverse limit exist. In addition, there are natural functions
$$\Delta_A\colon \varinjlim\Hom(A,X^i)\longrightarrow \Hom(A,\varinjlim X^i)$$
and
$$\nabla_A\colon \varinjlim\Hom(X_i,A)\longrightarrow \Hom(\varprojlim X_i,A),$$
again provided the existence of suitable limits. In addition,  if the bonding maps
in the direct sequence are monomorphisms then $\Delta$ is injective, and if the bonding maps in the inverse
sequence are epimorphisms then $\nabla$ is injective.

\begin{defn}
An object $A$ of the category $\Cc$ is \emph{small} if $\Delta_A$ is a bijection for every convergent direct
sequence in $\Cc$ with monomorphic bonding maps.
Dually, an object $A$ of the category $\Cc$ is \emph{co-small} if $\nabla_A$ is a bijection for every
convergent inverse sequence in $\Cc$ with epimorphic bonding maps.
\end{defn}

\begin{rmk}
Although small objects and related concepts are extensively studied in category theory, the terminology and
precise definitions are far from being standardized. They are known variously as finitely generated objects,
finitely presentable objects,
compact objects, small-projective objects, tiny objects etc., and the definitions may require preservation of
arbitrary co-limits, filtered co-limits, co-limits with monomorphic bonding maps, co-products and several
other variants. The dual concept is usually known as co-small or co-compact object, and there is also
a variety of definitions.
\end{rmk}

It is not difficult to see that small objects in the category of groups are precisely the finitely
generated groups. To understand which objects are co-small in the category of groups observe that
for a given inverse sequence $(H_i,p_i^j)$ the image of $\nabla_A$ consists precisely of homomorphisms
from $\varprojlim H_i$ to $A$ that can be factored through some projection $p_i\colon \varprojlim H_i\to H_i$.
Thus we obtain the following restatement of Theorem \ref{thm:lim-slender}

\begin{thm}
\label{thm:lim-slender as co-small}
The following statements are equivalent:
\begin{enumerate}
\item The group $G$ is a co-small object in the category of groups.
\item $G$ is an inverse slender group.
\item $\nabla_G\colon\varinjlim\Hom(F_n,G)\to \Hom(\hF,G)$ is a bijection.
\item For every inverse sequence of groups $(H_i,p_i^j)$ (non necessarily with surjective bonding maps)
$\nabla_G$ induces a bijection between $\Hom(\varprojlim H_i,G)$ and $\varinjlim\Hom(\im p_i,G)$.
\end{enumerate}
\end{thm}

\begin{rmk}\label{rmk:slender}
The class of groups satisfying the statements of the above theorem is by no means small. In fact, by
combining Proposition \ref{prop:nc implies inverse limit} with results of \cite{ConnerCorson2019} and
\cite{Corson2018} one can see that the above hold for any group in the following list:
\begin{enumerate}
\item Free groups,
\item Free abelian groups,
\item Torsion-free word hyperbolic groups,
\item Torsion-free one-relator groups,
\item Baumslag-Solitar groups,
\item Countable diagram groups over some semigroup presentation,
\item Thompson's group $F$.
\end{enumerate}

\end{rmk}

Other types of slender groups can be characterized in a similar way, which allows to compare them
in a unified way. So, for slender groups we have the following result, where the first statement is
essentially the definition and the second statement follows from \cite[Theorem 4.3]{Gobel1975}.

\begin{prop}
\label{prop:slender as co-small}
A group $G$ is slender, if, and only if
$$\nabla_G\colon\varinjlim\Hom(\ZZ^n,G)\to \Hom(\ZZ^\NN,G)$$
is a bijection, or equivalently,
if the natural inclusion
$$\bigcup_{i\in\NN} \Hom(\prod_{i=1}^n H_i,G)\to \Hom(\prod_{i\in\NN} H_i,G)$$
is  a bijection  for every sequence of groups $\{H_i\mid i\in\NN\}$.
\end{prop}

The characterization of n-slender groups is slightly more complicated, and it can be
expressed in terms of the restriction of homomorphisms along the inclusion
$\HEG\hookrightarrow\hF$.

\begin{prop}
\label{prop:n-slender as co-small}
A group $G$ is n-slender if, and only if,  the composition
$$\varinjlim\Hom(F_n,G)\stackrel{\nabla_G}{\longrightarrow}\Hom(\hF,G)\stackrel{\text{res}}{\longrightarrow}
\Hom(\HEG,G)$$
is surjective.
\end{prop}
\begin{proof}
If $G$ is n-slender, then for every homomorphism $\varphi\colon\HEG\to G$ there exists a homomorphism
$\overline\varphi\colon F_i\to G$, such that $\overline\varphi\circ p_i=\varphi$ or in other words,
$\text{res}(\nabla_G(\overline\varphi))=\varphi$. The converse statement is obvious.
\end{proof}

We can use the above to characterize n-slender groups among inverse limit slender groups.

\begin{cor}
Let $G$ be an inverse limit slender  group.
Then $G$ is n-slender if, and only if the restriction $\Hom(\hF,G)\to \Hom(\HEG,G)$
is a bijection.
\end{cor}
\begin{proof}
Since $G$ is inverse limit slender, $\nabla_G\colon\varinjlim\Hom(F_n,G)\to\Hom(\hF,G)$ is bijective
by Theorem \ref{thm:lim-slender as co-small} and
$\text{res}\colon \Hom(\hF,G)\to \Hom(\HEG,G)$ is injective by Theorem \ref{thm:res is injective}.
It is thus sufficient to observe that, by Proposition \ref{prop:n-slender as co-small}, $G$ is n-slender
if, and only if, $\text{res}$ is surjective.
\end{proof}

\section{Homology of Barratt-Milnor examples}
\label{sec:Homology of Barratt-Milnor examples}




Let $\HH_n$ be the one-point union of a sequence of spheres in $\RR^{n+1}$
whose radii converge to zero (so that $\HH_1$ is just the Hawaiian earring), and
let $\mathbb S^n$ denote the unit sphere in $\RR^{n+1}$.  Eda and Kawamura  \cite{EdaKawamura2000}
showed that for $n>1$ the group $\pi_n(\HH_n)$ is isomorphic to $\ZZ^\NN$ by a homomorphism taking $\mathbf e_i\in\ZZ^\NN$
to the homotopy class of embedding of  $\mathbb S^n$ onto the $i$-th sphere of $\HH_n$.

Barratt and Milnor \cite{BarrattMilnor1962} showed that $H_l(\HH_n)$ was uncountable for infinitely many $l$.   Here we will further investigate this group by showing that they are cotorsion.

\begin{defn}
  An abelian group $G$ is \emph{cotorsion} if it is a direct summand of every extension by a torsion-free group, i.e. $\operatorname{Ext}(A, G) =0$ for all torsion-free groups $A$.  
  A group $G$ is \emph{Higman-complete} if for any sequence $(g_i)$ of elements of $G$ and for a given sequence of words $(w_i)$ in two symbols, there exists a sequence $(h_i)$ of elements of $G$ such that all the equations $h_i = w_i(f_i, h_{i+1})$ hold simultaneously.
\end{defn}

Herfort and Hojka proved that the class of cotorsion groups is the same as the class of abelian Higman-complete groups  \cite[Theorem 3]{HerfortHojka2017}.

\begin{lem}\label{lem: isom}
For $l>n>1$, the retraction $q_k\colon \HH_n\to \HH_n$ that collapses the first $k$ spheres to the base
point induces the identity map on $H_l(\HH_n)$.
\end{lem}
\begin{proof}
Let $A:= \HH_n \backslash \{p_1,\cdots, p_k\}$ where $p_i$ is any point other than the base point
of the $i$-th sphere in $\HH_n$.  Notice that $A$ deformation retracts to the image of $q_k$.  Let $B:=
\bigcup\limits_{i=1}^k S_i\backslash x_0$ where $S_i$ is the $i$-th sphere in $\HH_n$  and $x_0$ is the
base point.  Then the Mayer-Vietoris sequence gives us the following exact sequence
$$\cdots H_l(A\cap B) \to H_l(A)\oplus H_l(B) \to H_l(A\cup B) \to H_{l-1}(A\cap B)\to \cdots. $$
Since $A\cap B$ is homotopy equivalent to the disjoint union of spheres with dimension $n-1$, we have that
$H_l(A\cap B)$, $H_{l-1}(A\cap B)$ are trivial.  Thus the middle homomorphism is an isomorphism.  Since
$H_l(B)$ is trivial, the middle isomorphism is induced by inclusion.  Thus the retract $q_k$ induces the
identity homomorphism.
\end{proof}

Let us denote by $h\colon \pi_l(\HH_n) \to H_l(\HH_n)$ the standard Hurewicz homomorphism.

\begin{lem}\label{lem: nice maps} Let $l>n>1$.
For every sequence of elements $(b_i)$ in $\pi_l(\HH_n)$ there exits a homomorphism
$\varphi\colon \ZZ^\NN \to H_l(\HH_n)$ such that $\varphi(\mathbf e_i) = h(b_i)$.
\end{lem}
\begin{proof}
Let $(b_i)$ be a sequence of elements in $\pi_l(\HH_n)$ and for each $i$ let $f_i\colon
\mathbb S^l \to \HH_n$ be a representative of $b_i$.  Then we can define a continuous map
$p\colon \HH_l \to \HH_n$ by $p|_{S_i} = q_i\circ f_i$, where $S_i$
is the $i$-th sphere of $\HH_l$.  Since $\im(q_i)$ converge to the wedge point, $p$ is a continuous map and
induces a homomorphism $p_*\colon \ZZ^\NN \to \pi_l(\HH_n)$.  If $\iota_i: \mathbb S^l \to \HH_l$ is the
embedding that corresponds to $\mathbf e_i\in \ZZ^\NN$ under the identification of
$\pi_l(\HH_l)$  with $\ZZ^\NN$, then
 \begin{align*}
    h\circ p_*\bigl([\iota_i]\bigr)  & = h\bigl([p\circ \iota_i]\bigr) = h\bigr([q_i\circ f_i]\bigr) \\
     & = h\circ q_{i*}\bigl([f_i]\bigr) = q_{i*}\circ h\bigl([f_i]\bigr) = h\bigl([f_i]\bigr)
  \end{align*}
where the equality follow from Lemma \ref{lem: isom} and from the naturality of $q_i$ and $h$.
Thus  $h\circ p_*$ sends $\mathbf e_i$ to $h(b_i)$ as claimed.
\end{proof}

\begin{thm}\label{thm: no hom}
For $l>n>1$ and let $G$ be a \limslender\ group. Then the homomorphism
$$\Hom\bigl(H_l(\HH_n), G\bigr)\longrightarrow\Hom\bigl(\pi_l(\HH_n), G\bigr),$$
induced by the Hurewicz homomorphism, is trivial.

In particular,  $\Hom\bigl(H_{n+1}(\HH_n), G\bigr) $ is trivial for all $n$.
\end{thm}
\begin{proof}
Suppose that $\psi\colon \pi_l(\HH_n)\to G$ is a nontrivial homomorphism to a \limslender~ group $G$
that factors through the Hurewicz homomorphism $h$. Then there exists a homomorphism
$\overline \psi\colon H_l(\HH_n)\to G$ such
that $\overline \psi\bigl(h(b)\bigr) \neq 1$ for some $b \in \pi_l(\HH_n)$.  Then  $\overline \psi\circ
\phi(\mathbf e_i)\neq 1$ for all $i$ where $\phi$ is the homomorphism from Lemma \ref{lem: nice maps} for
the constant sequence $(b)$. This contradicts the assumption that $G$ is \limslender.  Thus the Hurewicz
homomorphism induces the trivial homomorphism from $\Hom\bigl(H_l(\HH_n), G\bigr) $ to
$\Hom\bigl(\pi_l(\HH_n), G\bigr) $ for any \limslender~ group $G$.

The last statement follows from the well-known fact that the Hurewicz homomorphism
$h\colon\pi_{n+1}(\HH_n)\to H_{n+1}(\HH_n)$ is surjective.
\end{proof}

\begin{thm}\label{thm: cotorsion}
For  $l>n>1$, the image of the Hurewicz map from $\pi_l(\HH_n)$ to $H_l(\HH_n)$ is cotorsion.  In particular, $H_{n+1}(\HH_n)$ is cotorsion for $n>1$.
\end{thm}
\begin{proof}
By \cite[Section 2]{HerfortHojka2017} it is sufficient to show that for every sequence
of elements $(b_i)$ in $h\bigl(\pi_l(\HH_n)\bigr)$ and every
sequence of natural numbers $(n_i)$,  the infinite system of equations
$$x_i = b_i +n_{i}x_{i+1}\ \ \ i=1,2,3,\ldots$$
has a solution in $h\bigl(\pi_l(\HH_n) \bigr)$
(note that the above equations are the abelianization of  the Higman equations).
By Lemma \ref{lem: nice maps}, we can find a homomorphism
$\varphi\colon \ZZ^\NN \to H_l(\HH_n)$ such that $\phi(\mathbf e_i) = b_i$.
It is then easy to check that for
$a_i:=\bigl(\underbrace{0, \cdots, 0,}_{(i-1) - \text{times}} 1, n_i, n_in_{i+1},n_in_{i+1}n_{i+2}, \cdots
\bigr)$,
the sequence $\bigl(\varphi(a_i)\bigr)$ is a solution to the above system of equations in
$h\bigl(\pi_l(\HH_n)\bigr)$.
\end{proof}

As we already mentioned in the Introduction, the groups $H_l(\HH_n)$ are often very big.
The precise statement, proved by Barratt and Milnor \cite{BarrattMilnor1962} is that
the image of the rational Hurewicz homomorphism $\pi_l(\HH_n)\to H_l(\HH_n;\QQ)$ is uncountable
whenever $l\equiv 1\ \mathrm{mod} (n-1)$. Thus the above theorem (in particular the case $l=n+1$)
lead us to believe that more is true:

\begin{conj}
For  $l>n>1$ all groups  $H_l(\HH_n)$ are cotorsion.
\end{conj}

\section{Applications to shape groups and \v Cech cohomology}
\label{sec:Applications to shape groups and Cech cohomology}

Classical homotopy and (co)homology groups are best suited for the study of simplicial or CW-complexes.
For more general spaces like metric compacta or Peano continua with bad local properties these invariants
fail to give useful information about their homotopy type. Classical examples include
Warsaw circle, Sierpinski gasket, Hawaiian Earring, $p$-adic solenoids, attractors of dynamical
systems, boundaries of groups and many others. The main problem is that there may not be sufficiently many
maps from polyhedra (e.g. spheres) to such spaces to be able to distinguish between them.
A standard way around this problem
is to approximate a given space by a sequence of nicer spaces like polyhedra or absolute neighbourhood
retracts. This is the so-called \emph{shape theory} approach, for which we refer to the classical
monograph \cite{MardesicSegal1982} by Mardesic and Segal.

Our main interest will be in compact, connected metric spaces that are also locally connected.
Such spaces are classically known as \emph{Peano continua} and form a quite general class of spaces,
which are still sufficiently nice to allow a rich structure theory. For every Peano continuum $X$ there
exists an inverse sequence of finite polyhedra $(X_i,p_i^j)$, such that $X=\varprojlim(X_i,p_i^j)$.
These approximations are unique up to a homotopy equivalence of inverse sequences
(see \cite[Appendix 1]{MardesicSegal1982}) and can be used to define homotopy invariants of $X$.

Let $x_0$ be a base-point in $X=\varprojlim X_i$ and let $k\ge 1$. Then for every $i$ take
$p_i(x_0)$ as a base-point in $X_i$ and define the $k$-th \emph{shape homotopy group} of $X$ as
$\check\pi_k(X,x_0):=\varprojlim \pi_k(X_i,p_i(x_0))$. This definition make sense for $k=0$ as well,
but $\check\pi_0(X,x_0)$ is in general only a pointed set. The \emph{shape fundamental group}
$\check\pi_1(X,x_0)$ can be a non-commutative group, while all higher shape homotopy groups are abelian.
Note that there is a natural homomorphism $\pi_1(X,x_0)\to \check\pi_1(X,x_0)$ which is often injective
(e.g., if $X$ is 1-dimensional),
and that for $\pi_1(X,x_0)$ there holds a famous Shelah's dichotomy \cite{Shelah1988}: they are either
finitely presented or uncountable. In spite of that we have the following result.

\begin{thm}\label{cor:hom check}
Let $X$ be a Peano continuum with a representation $X=\varprojlim X_i$ as a limit of an inverse sequence
of finite polyhedra,  and let $G$ be an inverse limit slender group.
Then $\Hom\bigl( \check{\pi}_1(X),G\bigr) = \varinjlim \Hom\bigl(\pi_1(X_i),G\bigr)$.
In particular, if $G$ is countable, then $\Hom\bigl( \check{\pi}_1(X),G\bigr)$ is also countable.
\end{thm}
\begin{proof}
Since $X$ is locally path-connected, the inverse sequence of fundamental groups $(\pi_1(X_i),(p_i^j)_\sharp)$
satisfies the Mittag-Leffler condition (\cite[Sec. II,7.2]{MardesicSegal1982}, see also
\cite[Cor. 3.2]{ConnerHerfortKentPavesic2}). Then Theorem \ref{thm:lim-slender as co-small}(4)
implies the formula $\Hom\bigl( \check{\pi}_1(X),G\bigr) = \varinjlim \Hom\bigl(\pi_1(X_i),G\bigr)$.
Finally, if $G$ is countable, then $\Hom(\pi_1(X_i),G)$ are countable for every $i$ (because
fundamental groups of finite polyhedra are finitely presented), and so their inverse limit
$\Hom(\check\pi_1(X),G)$ must be countable as well.
\end{proof}

The last theorem can be extended to higher shape groups for compact and connected metric spaces that satisfy
a local connectivity condition called $LC^n$ for suitable $n$ (see \cite[Sec. II,7.2]{MardesicSegal1982}).
We leave the details to the reader.

Our final application is a universal coefficient theorem for \v Cech cohomology groups with coefficients
in a slender abelian group.
Every compact metric space $X$ can be represented as a limit of an inverse sequence of finite polyhedra
$X=\varprojlim X_i$. Such representation are unique up to homotopy equivalence of inverse sequences,
so one can define \emph{\v Cech homology} groups with coefficients in an abelian group $G$ as
$$\check H_n(X,G):=\varprojlim H_n(X_i;G).$$
Since the cohomology is a contravariant functor, it turns inverse sequences into direct sequences, so
one can also define \emph{\v Cech cohomology} groups with coefficients in $G$ as
$$\check H^n(X,G):=\varinjlim H^n(X_i;G).$$
\v Cech cohomology groups are a tool of choice for the study of spaces with bad local properties.
On the other hand, \v Cech homology is much less used, because in general it does not form long exact
sequences of a pair, and is thus not a homology theory in the usual sense (note however, that
it satisfies the exactness axiom if we take coefficients in a finite group or in a field
\cite[Ch. IX]{EilenbergSteenrod1952}). Nevertheless, we may use our results on inverse limit slender
groups to derive a Universal coefficient theorem for \v Cech cohomology.

We first need some algebraic lemmas.

\begin{lem}\label{lem: minimal free factor}
Let $H$ be a finitely generated free abelian group and $B$ a subgroup.  Then there exists a unique maximal
subgroup $C$ of $H$ containing $B$ such that $C/B$ is finite and $C$ is a free abelian factor of $H$.
\end{lem}

\begin{proof}[Sketch of proof.] Let $C = \bigl\{ a\in H\mid a^k\in B \text{ for some } k\in \mathbb Z\backslash\{0\} \bigr\}$.  It is trivial to show that $C$ is actually a subgroup.  Since $C/B$ is an abelian, finitely generated, torsion group it is finite. It is an exercise to see that if $a^l \in C$ for some $a\in H$ and some nonzero integer $l$, then $a \in C$. Thus $H/C$ is a free abelian group and $C$ is a free abelian factor of $H$.

Suppose that $D/B$ is finite.  Then for every $d\in D$ there exits a $k\in \mathbb N$ such that $d^k \in B$ which implies that $D\subset C$.
\end{proof}

\begin{lem}\label{lem: free abelian}
  Every homomorphism from an inverse limit of finitely generated abelian groups to a \limslender~ group factors through an inverse limit of finitely generated free abelian groups.
\end{lem}

\begin{proof}
Let $\widehat H = \varprojlim(H_i, p_i^j)$ where $(H_i, p_i^j)$ is an inverse sequence of finitely generated
abelian groups.  Let $T_i$ be the torsion subgroup of $H_i$ and $q_i: H_i \to A_i$ be the quotient homomorphism
with kernel $T_i$. Then $A_i$ is free abelian. Let $(A_i, \bar p_i^j)$ be the inverse sequence of free abelian
groups defined by $\bar p_i^j \circ q_i = q_j \circ p_i^j$.  Then the maps $\{q_i\}$ induce a homomorphism
$q:\widehat H \to \widehat A$ where $\widehat A = \varprojlim (A_i, \bar p_i^j)$.  The short exact sequence
$0\rightarrow T_i \rightarrow H_i \stackrel{q_i}{\rightarrow} A_i \rightarrow 0$ gives rise to the exact
sequence
$0 \rightarrow \widehat T \rightarrow \widehat H \stackrel{q}{\rightarrow} \widehat A \rightarrow \varprojlim^1
(T_i, p_i^j|_{T_i})$ where $\widehat T = \varprojlim(T_i, p_i^j|_{T_i})$.
Then by \cite[Proposition 11.3.13]{Geoghegan1986} the last term is $0$, so $q$ is surjective with kernel
$\widehat T$.

Let $\varphi: \widehat H \to G$ be a homomorphism to a \limslender~ group.  It is immediate that
$\widehat H = \varprojlim (p_i(H_i), p_i^j|_{p_i(H_i)})$.  Since $G$ is \limslender, there exists an $m$ such
that $\ker (p_{m}) \subset \ker (\varphi)$.
Since \slender groups are torsion-free, $\ker (q_{m}\circ p_{m}) \subset\ker(\varphi)$.  Thus we have that
$\ker(q) = \widehat T\subset \ker(q_{m}\circ p_{m})$, therefore $\varphi$ factors through $q$.
\end{proof}

\begin{lem}\label{lem:without hypothesis}
Let $(H_i, p_i^j)$ be an inverse sequence of finitely generated abelian groups.  Then every homomorphism from
$\varprojlim(H_i, p_i^j)$ to a \limslender~ group $G$ factors through a projection, and the
natural homomorphism
$$\nabla_G\colon \varinjlim \Hom\bigl(H_i,G\bigr) \longrightarrow\Hom\bigl(\varprojlim H_i,G\bigr)$$
is bijective.
\end{lem}
\begin{proof}
By Lemma \ref{lem: free abelian}, we may assume that $(H_i, p_i^j)$ is an inverse sequence of finitely generated
free abelian groups.  Let $\widehat H = \varprojlim(H_i, p_i^j)$ and suppose that $\varphi: \widehat H \to G$
is a homomorphism to an \limslender~ group $G$.  Then there exist an $m_0$ and
$\varphi': p_{m_0}(\widehat H) \to G$ such that $\varphi = \varphi'\circ p_{m_0}$.

By Lemma \ref{lem: minimal free factor}, there exists $H_i'$ a free abelian factor of $H_i$ such that
$p_i(\widehat H) $ is a finite index subgroup  of $ H_i'$.  Let $q_i: H_i \to H_i'$ be the quotient homomorphism
projecting $H_i$ onto the factor $H_i'$.  Notice that $\widehat H = \varprojlim(H_i', p_i^j|_{H_i'})$,
since $p_i(\widehat H) \subset H_i'$.

For any $i\geq j\geq k$, we have that
$$\bigl[H_k', p_{j,k}(H_j')\bigr]\leq \bigl[H_k', p_{i,k}(H_i')\bigr]\leq
\bigl[H_k', p_{k}(\widehat H)\bigr]< \infty.$$  Thus for each $k$, we may choose $n_k\geq k$ such that
$\bigl[H_k', p_{n_k,k}(H_{n_k}')\bigr]= \bigl[H_k', p_{i,k}(H_i')\bigr]$ for all $i\geq n_k$.  Then, for
$i\geq n_k$, the natural quotient map from $H_k'/p_{i,k}(H_{i}')$ to $H_k'/p_{n_k,k}(H_{n_k}')$ is injective
since the groups have the same finite cardinality.  Thus $p_{i,k}(H_{i}')= p_{n_k,k}(H_{n_k}')$ for all
$i\geq n_k$.  It is then an exercise to show that $p_{n_k,k} (H_{n_k}') = p_{k} (\widehat H)$ for any $k$.
In particular, $p_{n_{m_0},{m_0}} (H_{n_{m_0}}') = p_{m_0} (\widehat H)$, which allows us to define
$\bar \varphi: H_{n_{m_0}} \to G$ by $ \bar \varphi =\varphi'\circ p_{n_{m_0}, m_0}\circ q_{n_{m_0}}$.  Since
$q_{n_{m_0}}|_{p_{n_{m_0}}(\widehat H)}$ is the identity it is immediate that
$\varphi = \bar \varphi\circ p_{n_{m_0}}$.
\end{proof}

The main point of the last Lemma is that in some cases the assumption that an inverse sequence
is Mittag-Leffler is not needed. This is of crucial importance in the next theorem, which is the main
result of this section.

\begin{thm}[Universal Coefficients Theorem for \v Cech Cohomology]\label{thm: uct}
Let $X $ be a compact metric space, represented as a limit of an inverse sequence of finite polyhedra,
$X = \varprojlim X_i$. Then for every slender abelian group we have short exact sequences
$$ 0 \longrightarrow \lim\limits_{i \rightarrow\infty} \operatorname{Ext}\bigl(H_{n-1}(X_i), G\bigr)
\longrightarrow \check{H}^n(X; G) \longrightarrow \Hom(\check{H}_n(X),  G)\longrightarrow 0$$
\end{thm}
\begin{proof}
Notice that $(H_n(X_i), p_{i,j*})$ is an inverse sequence of finitely generated abelian groups.  Thus by Lemma
\ref{lem:without hypothesis}, every homomorphism from their inverse limit to $G$ factors through a projection.

The theorem now follows  by taking the direct limit of the following short exact sequence
$$ 0 \longrightarrow  \operatorname{Ext}\bigl(H_{n-1}(X_i), G\bigr)\longrightarrow {H}^n(X_i; G) \stackrel{h}
{\longrightarrow} \Hom(H_n(X_i),  G)\longrightarrow 0$$
\end{proof}

We conclude with an example that shows that the above result is not valid for groups that are not slender.
Let $B_i$ be a bouquet of $i$ circles.  Then there is a natural map for $B_i$ to $B_{i-1}$ that collapses one of
the circles and is a homeomorphism on the rest of the circles.  The inverse limit of this system is the Hawaiian
earring, which we will denote by $\HH_1$.  Notice that $\operatorname{Ext}\bigl(H_{n-1}(B_i), G\bigr)= 0$
since $H_{n-1}(B_i)$ is free abelian or trivial.

\begin{prop}
If $G$ is a non-slender abelian group $G$, then the homomorphism
$$\nabla_G\circ\overline h\colon \check{H}^1(\mathcal H_1; G) \longrightarrow
\Hom(\check{H}_1(\mathcal H_1,G))$$
is not surjective, so it cannot be extended to a short exact sequence as in Theorem \ref{thm: uct}.
\end{prop}
\begin{proof}
Since $H_0(B_i) = \mathbb Z$ it follows that $\operatorname{Ext}\bigl(H_{0}(B_i), G\bigr) = 0$.
As before we can take the direct limit of the short exact sequences
$$0 \longrightarrow 0\longrightarrow H^1(B_i; G) \stackrel{h_i}{\longrightarrow} \Hom(H_1(B_i),  G)
\longrightarrow 0$$
to obtain the short exact sequence,
$$0 \longrightarrow  0\longrightarrow \check{H}^1(\mathcal H_1; G) \stackrel{h}{\longrightarrow}
\lim\limits_{i
\rightarrow\infty}\Hom(H_1(B_i),  G)\longrightarrow 0.$$

Notice that $\varinjlim\Hom(H_1(B_i),  G) = \varinjlim\Hom(\ZZ^i, G)$ and
$\Hom(\varprojlim H_1(B_i),  G)=\Hom\bigl(\ZZ^\NN,  G\bigr)$.
Thus the natural map $\nabla_G\colon\varinjlim\Hom(H_1(B_i), G)\to\Hom(\check{H}_1(\mathcal H_1),  G)$
is surjective if and only if the abelian group $G$ is \limslender.
Therefore
$$ \check{H}^1(\mathcal H_; G) \stackrel{\nabla_G\circ\overline h}{\longrightarrow}
\Hom(\check{H}_1(\mathcal H_1), G)$$
is not surjective.
\end{proof}


\begin{thebibliography}{99}

\bibitem{Baer1937}
R. Baer, \emph{Abelian groups without elements of finite order}, Duke Math. J., {\bf 3} (1937), 68--122.

\bibitem{BarrattMilnor1962}
M. G. Barratt, J. Milnor, \emph{An example of anomalous singular homology}, Proc. Amer.
Math. Soc., {\bf 13} (1962), 293--297.

\bibitem{CannonConner2006}
J. W. Cannon, G. R. Conner, \emph{On the fundamental groups of one-dimensional spaces},
Topology Appl., {\bf 153} (2006), 2648--2672.

\bibitem{Blass1994}
A. Blass, \emph{Cardinal characteristics and the product of countably many infinite cyclic groups},
J. Algebra {\bf 169} (1994), 512--540.

\bibitem{ConnerCorson2019}
G.R. Conner, S.M. Corson, \emph{A note on automatic continuity}, Proc. Amer. Math. Soc. {\bf 147}
(2019) 1255--1268.

\bibitem{ConnerHerfortKentPavesic1}
G.R. Conner, W. Herfort, C. Kent, P. Pave\v{s}i\'{c}, \emph{The path space and covering fibrations},
in preparation.

\bibitem{ConnerHerfortKentPavesic2}
G.R. Conner, W. Herfort, C. Kent, P. Pave\v{s}i\'{c}, \emph{Uncountable groups and the geometry of inverse
limits of coverings}, to appear.

\bibitem{ConnerKent2012}
G.R. Conner, C. Kent, \emph{Inverse limits of finite rank free groups}, J. Group Theory {\bf 15} (2012),
823--829.

\bibitem{Corson2018}
S.M. Corson, \emph{Root extraction in one-relator groups and slenderness}, Comm. Algebra, {\bf 46} (2018),
4317--4324.

\bibitem{CorsonKazachkov2020}
S. M. Corson,  I. Kazachkov, \emph{On preservation of automatic continuity}, Monatsh.
Math., {\bf 191} (2020), 37--52.

\bibitem{CorsonVarghese2020}
S.M. Corson, O. Varghese,\emph{A Nunke type classification in the locally compact setting},
J. Algebra {\bf 563} (2020), 49--52.

\bibitem{Dudley1961}
R. Dudley, \emph{Continuity of homomorphisms}, Duke Math. J. {\bf 28} (1961), 587--594.

\bibitem{Eda1992}
K. Eda., \emph{Free $\sigma$-products and non-commutatively slender groups}, J. Algebra, {\bf 148} (1992),
243--263.

\bibitem{Eda1998}
K. Eda, \emph{The non-commutative Specker phenomenon}, J. Algebra, {\bf 204} (1998), 95--108.

\bibitem{EdaKawamura2000}
K. Eda, K. Kawamura, \emph{Homotopy and homology groups of the n-dimensional Hawaiian earring},
Fund. Math., {\bf 165} (2000), 17--28.

\bibitem{EdaNakamura2013}
K. Eda, J. Nakamura, \emph{The classification of the inverse limits of sequence of free groups of finite rank},
Bull. Lond. Math. Soc. {\bf 45} (2013), 671--676.

\bibitem{Fuchs1958}
L. Fuchs, \emph{Abelian groups}, Publishing House of the Hungarian Academy of Sciences, Budapest,
1958.

\bibitem{Fuchs2015}
L. Fuchs, \emph{Abelian groups}, Springer Monographs in Mathematics. Springer, Cham, 2015.

\bibitem{EilenbergSteenrod1952}
S. Eilenberg, N.E. Steenrod, \emph{Foundations of Algebraic Topology}, Princeton mathematical series {\bf 15},
(Princeton University Press, Princeton, 1952).

\bibitem{Geoghegan1986}
R. Geoghegan, \emph{The shape of a group—connections between shape theory and the homology of groups},
In Geometric and algebraic topology, vol. 18 of Banach Center Publ., pp. 271--280. PWN, Warsaw, 1986.

\bibitem{Geoghegan2008}
R. Geoghegan, \emph{Topological methods in group theory}, Graduate Texts
in Mathematics 243 (Springer, New York, 2008)

\bibitem{Gobel1975}
R. G\"obel, \emph{Stout and slender groups}, J. Algebra {\bf 35} (1975), 39--55.

\bibitem{HerfortHojka2017}
W. Herfort, W. Hojka, \emph{Cotorsion and wild homology}, Israel J. Math. {\bf 21} (2017), 275--290.

\bibitem{Kryuchkov2007}
N. I. Kryuchkov, \emph{Some generalizations of slender abelian groups}, Fundam. Prikl. Mat.,
{\bf 13} (2007), 97--106.

\bibitem{MardesicSegal1982}
S.~Marde\v si\'c, J.~Segal, \emph{Shape Theory}, North Holland Mathematical Library, Vol. 26 (North Holland
Publishing Company, Amsterdam, 1982).

\bibitem{Nunke1961}
R.J. Nunke, \emph{Slender groups}, Bull. Amer. Math. Soc. {\bf 67} (1961), 274--275;
Acta Sci. Math. Szeged {\bf 23} (1962), 67--73.


\bibitem{Sasiada1959}
E. S\k{a}siada, \emph{Proof that every countable and reduced torsion-free abelian group is slender},
Bull. Acad. Polon. Sci. {\bf 7} (1959), 143--144.

\bibitem{Shelah1988}
S.~Shelah, \emph{Can the fundamental (homotopy) group of a space be
the rationals?}, Proc. AMS, {\bf 103} (1988), 627-632.

\bibitem{Slutsky2013}
K. Slutsky, \emph{Automatic continuity for homomorphisms into free products}, J. Symbolic Logic
{\bf 78} (2013), 1288--1306.

\bibitem{Specker1950}
E. Specker, \emph{Additive Gruppen von Folgen ganzer Zahlen}, Portugal. Math. {\bf 9} (1950), 131-–140.

\end{thebibliography}
\end{document}